\documentclass[12pt]{article}

\usepackage{amsfonts}
\usepackage{amsmath}
\usepackage{amssymb}
\usepackage{color}
\usepackage{amsthm}
\usepackage{hyperref}
\usepackage{makecell}

\newtheorem{theorem}{Theorem}
\newtheorem{corollary}[theorem]{Corollary}

\newtheorem{lemma}[theorem]{Lemma}
\newtheorem{proposition}[theorem]{Proposition}

\title{Biasing with an independent increment:\\ Gaussian approximations and proximity of Poisson mixtures} \author{Fraser Daly\footnote{Department of Actuarial Mathematics and Statistics, and the Maxwell Institute for Mathematical Sciences, Heriot--Watt University, Edinburgh EH14 4AS, UK.  E-mail: F.Daly@hw.ac.uk}} \date{\today}

\begin{document}

\maketitle

\noindent{\bf Abstract} 
By exploiting the well-known observation that size-biasing or zero-biasing an infinitely divisible random variable may be achieved by adding an independent increment, combined with tools from Stein's method for compound Poisson and Gaussian approximations, we establish three sets of approximation results: (a) bounds on the proximity of Poisson mixtures with infinitely divisible mixing distributions, (b) central limit theorems with explicit error bounds for sums of associated or negatively associated random variables which do not require boundedness of the underlying distributions, and (c) a Gaussian approximation theorem under a vanishing third moment condition. These exploit biasing by an independent increment directly, via an intermediate compound Poisson approximation, and through a convex ordering argument, respectively. Applications include a Dickman-type limit theorem, simple random sampling and urn models with overflow. 
\vspace{12pt}

\noindent{\bf Key words and phrases:} size-biasing; zero-biasing; compound Poisson; infinite divisibility; central limit theorem; (negative) association; Stein's method

\vspace{12pt}

\noindent{\bf MSC 2020 subject classification:} 60F05; 60E07; 60E15; 62E17

\section{Introduction}\label{sec:intro}

There are well-known connections between infinite divisibility and several of the coupling techniques that are employed frequently in conjunction with Stein's method for probability approximation, particularly in Poisson, compound Poisson and Gaussian approximation results. Two such couplings are the size- and zero-biased couplings; see \eqref{eq:sb_def} and \eqref{eq:zb_def} below for definitions of these concepts. In Gaussian approximation results, the proximity of a real-valued random variable $W$ to Gaussian can be measured by the proximity of $W$ and its zero-biased version: \cite{chen11} gives an overview and many applications of this approach. Similarly, explicit error bounds in Poisson and compound Poisson approximations can be derived in many applications via construction of the size-biased version of the underlying random variable. See \cite{barbour92_book} for a detailed discussion of the Poisson case, and \cite{barbour92_paper} for an introduction to Stein's method for compound Poisson approximation and \cite{daly13} for a discussion of connections with size-biasing. We will also make use of non-zero-biasing, defined in \eqref{eq:nz_def} below, an extension of zero-biasing beyond the case of mean-zero random variables. 

It is well-known that either size-biasing or non-zero-biasing a compound Poisson random variable can be achieved by adding an independent increment (see \cite{steutel73} and \cite{daly22}, respectively), and that a mean-zero, infinitely divisible random variable may be zero-biased by adding an independent increment \cite{goldstein24}. While these observations on biasing with an independent increment are familiar, they remain under-exploited in applications of Stein's method. Our overall aim in this note is to demonstrate how these observations can be employed in conjunction with Stein's method through three applications. Although these applications are on the surface rather different, they each employ similar constructions and illustrate a way of exploiting biasing by the addition of an independent increment.

The three applications we will consider are the following:
\begin{enumerate}
\item[(a)] Bounds on the proximity of two mixed Poisson distributions with infinitely divisible mixing distributions, which will be applied to give total variation error bounds in the approximation of certain Bernoulli-weighted sums of independent Poisson random variables by a mixed Poisson with Dickman mixing distribution. Here we use Stein's method for compound Poisson approximation.
\item[(b)] Explicit error bounds in Gaussian approximation for sums of associated or negatively associated random variables without the boundedness assumptions required by earlier results, such as the univariate approximation results established by Goldstein and Wiroonsri \cite{goldstein18} and Wiroonsri \cite{wiroonsri18}. We give applications to simple random sampling and urn models with overflow. This part of our work uses both compound Poisson and Gaussian approximation results from Stein's method.
\item[(c)] A Gaussian approximation bound for a random variable whose first three moments match those of the approximating Gaussian distribution, illustrated by a straightforward application to Student's \emph{t} distribution. This uses Stein's method for Gaussian approximation.
\end{enumerate}

A natural way to exploit ideas of biasing with an independent increment is in the compound Poisson approximation of another compound Poisson distribution (where, as we will see below, proximity can be measured by the distance between the increments obtained when size-biasing each) or in the Gaussian approximation of a mean-zero random variable (where proximity can be measured by the size of the increment obtained when zero-biasing). This is the idea underlying both (a) and (c) above. In (a) we make use of the fact that a mixed Poisson random variable with an infinitely divisible mixing distribution is compound Poisson. In (c) we use this zero-biasing approach to establish a convex ordering that yields a Gaussian approximation bound for infinitely divisible random variables whose first three moments match those of the approximating Gaussian distribution.

Our approach to exploiting biasing with an independent increment in (b) above is less direct. Here we use the fact that our random variable $W$ of interest has a distribution close to compound Poisson; we use results from \cite{daly13} to bound the distance of $W$ from a suitably chosen compound Poisson distribution, and then use an approach based on non-zero-biasing with an independent increment to quantify the distance of this compound Poisson distribution from Gaussian. As we will see, compared to a more direct approach in this setting, our method bypasses some technical difficulties associated with compound Poisson approximation via Stein's method.

The remainder of the paper is structured as follows. We will use Section \ref{sec:cp} to give brief introductions to Stein's method for compound Poisson approximation and the notions of size- and zero-biasing in order to provide the necessary background for the remainder of the paper. In Section \ref{sec:pois} we establish bounds on the proximity of mixed Poisson distributions with infinitely divisible mixing distributions. Central limit theorems with explicit error bounds for sums of associated or negatively associated random variables are derived in Section \ref{sec:normal}. Finally, in Section \ref{sec:ThirdMom} we present a Gaussian approximation result under the assumption of three matching moments, i.e., a vanishing third moment assumption in the standardised case.

\section{Technical preliminaries: Stein's method for compound Poisson approximation}\label{sec:cp}

In this section we give an introduction to size- and zero-biasing, and to Stein's method for compound Poisson approximation. These techniques will be employed in each of the three application areas we discuss in Sections \ref{sec:pois}--\ref{sec:ThirdMom}. In the work that follows we will also need to use Stein's method for Gaussian approximation; this will be introduced as needed in Sections \ref{sec:normal} and \ref{sec:ThirdMom} below.

Given a non-negative random variable $X$ with finite, non-zero mean, its size-biased version, which we denote here by $X^\text{s}$, is defined by
\begin{equation}\label{eq:sb_def}
\mathbb{E}f(X^\text{s})=\frac{\mathbb{E}[Xf(X)]}{\mathbb{E}X}\,,
\end{equation}    
for any function $f:\mathbb{R}^+\to\mathbb{R}$ for which the expectation on the right-hand side exists. We recall that $X$ (or, equivalently, its distribution) is said to be infinitely divisible if, for each $n=1,2,\ldots$, there are independent and identically distributed (i.i.d.) random variables $X_1^{(n)},\ldots,X_n^{(n)}$ such that $X$ is equal in distribution to $X_1^{(n)}+\cdots+X_n^{(n)}$. The observation that $X$ is infinitely divisible if and only if $X^\text{s}$ is equal in distribution to $X+Y$ for some random variable $Y$ which is independent of $X$ is due to Steutel \cite{steutel73}, though he neither uses the language of size-biasing nor requires that $X$ have finite mean. We refer to the survey \cite{arratia19} for a modern discussion of Steutel's result and its applications.

In our work we will employ Steutel's result in the case where $X$ is a compound Poisson random variable on $\mathbb{Z}^+=\{0,1,\ldots\}$, that is, when $X$ may be written in the form $\xi_1+\cdots+\xi_N$, where $\xi,\xi_1,\xi_2,\ldots$ are i.i.d$.$ positive integer-valued random variables, and $N\sim\text{Pois}(\lambda)$ has a Poisson distribution with mean $\lambda>0$. For $j=1,2,\ldots$ we will write $\lambda_j=\lambda\mathbb{P}(\xi=j)$ for the parameters of this compound Poisson distribution. It is well known that a distribution on the non-negative integers with a positive mass at zero is infinitely divisible if and only if it is compound Poisson. Noting that a mixed Poisson distribution with an infinitely divisible mixing distribution is itself infinitely divisible \cite{karlis05} will allow us to use its compound Poisson representation in order to derive bounds on distances between such distributions in Section \ref{sec:pois}.

We will make use of tools from Stein's method for compound Poisson approximation, as developed by \cite{barbour92_paper}. This technique yields explicit error bounds in compound Poisson approximation. Suppose we have the compound Poisson random variable $X$ above, and another non-negative integer-valued random variable $W$ which we have reason to suppose is close to $X$ in distribution. We begin by defining the function $f:\mathbb{Z}^+\to\mathbb{R}$ via the Stein equation
\begin{equation}\label{eq:CPsteinop}
h(k)-\mathbb{E}h(X)=\sum_{i=1}^\infty i\lambda_if(k+i)-kf(k)
\end{equation}
for a given function $h:\mathbb{Z}^+\to\mathbb{R}$, and where $f(0)=0$. For a given class of functions $\mathcal{H}$ we may then write
\begin{equation}\label{eq:CPstein}
\sup_{h\in\mathcal{H}}|\mathbb{E}h(W)-\mathbb{E}h(X)|=\sup_{h\in\mathcal{H}}\left|\sum_{i=1}^\infty i\lambda_i\mathbb{E}f(W+i)-\mathbb{E}[Wf(W)]\right|\,.
\end{equation}
The left-hand side of this equation serves as our measure of proximity of $W$ and $X$. For example, choosing the class $\mathcal{H}$ equal to
\begin{itemize}
\item $\mathcal{H}_{\text{TV}}=\{\mathbf{1}(A):A\subseteq\mathbb{Z}^+\}$, where $\mathbf{1}(\cdot)$ is an indicator function, gives the total variation distance, in which case we denote the left-hand side of \eqref{eq:CPstein} by $d_\text{TV}(W,X)$.
\item $\mathcal{H}_{\text{K}}=\{\mathbf{1}(\cdot\leq x):x\in\mathbb{R}\}$ gives the Kolmogorov (or uniform) distance, denoted by $d_\text{K}(W,X)$.
\item $\mathcal{H}_{\text{W}}=\{h:|h(x)-h(y)|\leq|x-y|\}$ gives the Wasserstein distance, denoted by $d_\text{W}(W,X)$. Note that for real-valued $W$ and $X$ we have $d_\text{W}(W,X)=\int_\mathbb{R}|\mathbb{P}(W\leq x)-\mathbb{P}(X\leq x)|\,\text{d}x$.
\item $\mathcal{H}_2=\{h:|h(x)-h(y)|\leq|x-y|,|h^\prime(x)-h^\prime(y)|\leq|x-y|\}$ gives the Wasserstein-2 distance, denoted by $d_2(W,X)$.
\end{itemize}
We bound these distances by instead bounding the right-hand side of \eqref{eq:CPstein}, taking advantage of the fact that we do not need to consider a coupling of $W$ and $X$ when doing so. In order to bound the right-hand side of \eqref{eq:CPstein} we need to control the behaviour of $f$. It is typically sufficient to have bounds on the quantities
\begin{equation}\label{eq:SteinFactorDef}
M_0^{(X)}(\mathcal{H})=\sup_{h\in\mathcal{H}}\lVert f\rVert,\text{ and }M_1^{(X)}(\mathcal{H})=\sup_{h\in\mathcal{H}}\lVert \Delta f\rVert\,,
\end{equation}
where $\lVert f\rVert=\sup_{x}|f(x)|$ and $\Delta f=f(k+1)-f(k)$ for any function $f$. Unfortunately, good bounds on these quantities are available only under relatively restrictive assumptions on the parameters $\lambda_i$ of $X$. For example, Barbour et al$.$ \cite[Theorem 4]{barbour92_paper} show that 
\[
M_0^{(X)}(\mathcal{H}_{\text{TV}}),M_1^{(X)}(\mathcal{H}_{\text{TV}})\leq\min\left\{1,\frac{1}{\lambda_1}\right\}e^\lambda\,,
\]
and that this dependence on $\lambda$ cannot be improved in general. Such bounds are useful only when $\lambda$ is small. In some cases bounds of a better order are available. Under the assumption that $j\lambda_j\geq(j+1)\lambda_{j+1}$ for each $j$, Barbour et al$.$ \cite[Theorem 5]{barbour92_paper} give improved bounds, and similar improvements are available under the alternative assumption that 
\begin{equation}\label{eq:bx99}
\frac{\sum_{j=1}^\infty j(j-1)\lambda_j}{\sum_{j=1}^\infty j\lambda_j}<\frac{1}{2}\,;
\end{equation}
see Theorem 2.5 of \cite{barbour99}. We will discuss these improved bounds in more detail in Section \ref{sec:pois} below, where we will make use of them when deriving our mixed Poisson results.  

The form of the equation \eqref{eq:CPsteinop} can be motivated by noting that if $\mathbb{E}W=\mathbb{E}X=\lambda\mathbb{E}\xi$, then
\[
\sum_{i=1}^\infty i\lambda_i\mathbb{E}f(W+i)-\mathbb{E}[Wf(W)]=\mathbb{E}W\left[\mathbb{E}f(W+\xi^\text{s})-\mathbb{E}f(W^\text{s})\right]\,,
\]
which is easily shown to be zero if $W$ has the same distribution as $X$, thus explicitly identifying the independent increment $Y$ obtained when size-biasing $X$ as the size-biased version of the summand $\xi$. 

The form of \eqref{eq:CPsteinop} can be similarly motivated by using zero-biasing in place of size-biasing. For a random variable $X$ with zero mean and positive, finite variance, the zero-biased version $X^\text{z}$ is defined by
\begin{equation}\label{eq:zb_def}
\mathbb{E}f^\prime(X^\text{z})=\frac{\mathbb{E}[Xf(X)]}{\mathbb{E}[X^2]}\,,
\end{equation} 
for all Lipschitz functions $f:\mathbb{R}\to\mathbb{R}$ for which the expectation on the right-hand side exists. This definition was introduced by Goldstein and Reinert \cite{goldstein97} in the context of Gaussian approximation by Stein's method, noting that the mean-zero Gaussian distribution is the unique fixed-point of this transformation. D\"obler \cite{dobler17} establishes the existence of two generalisations of zero-biasing which relax the restriction of $X$ having mean zero. For a random variable $X$ with $\mathbb{E}[X^2]>0$ we define
\begin{itemize}
\item the non-zero-biased version of $X$, denoted by $X^\text{nz}$, by
\begin{equation}\label{eq:nz_def}
\mathbb{E}f^\prime(X^\text{nz})=\frac{\mathbb{E}[(X-\mathbb{E}X)f(X)]}{\text{Var}(X)}\,,
\end{equation}
for all Lipschitz functions $f:\mathbb{R}\to\mathbb{R}$ for which the expectation on the right-hand side exists.
\item the generalised-zero-biased version of $X$, denoted by $X^\text{gz}$, by
\begin{equation}\label{eq:gz_def}
\mathbb{E}f^\prime(X^\text{gz})=\frac{\mathbb{E}[X(f(X)-f(0))]}{\mathbb{E}[X^2]}\,,
\end{equation}
for all Lipschitz functions $f:\mathbb{R}\to\mathbb{R}$ for which the expectation on the right-hand side exists.
\end{itemize}
It is known that if $X$ has the compound Poisson distribution defined above then $X^\text{nz}$ is equal in distribution to $X+\xi^\text{gz}$, where these two summands are independent; see Lemma 2.4 of \cite{daly22}. By the definitions \eqref{eq:nz_def} and \eqref{eq:gz_def}, and noting that $\text{Var}(X)=\lambda\mathbb{E}[\xi^2]$, we therefore have
\begin{align*}
\mathbb{E}[(X-\mathbb{E}X)f(X)]&=\text{Var}(X)\mathbb{E}f^\prime(X+\xi^\text{gz})=\frac{\text{Var}(X)}{\mathbb{E}[\xi^2]}\mathbb{E}[\xi(f(X+\xi)-f(X))]\\
&=\lambda\mathbb{E}[\xi(f(X+\xi)-f(X))]=\lambda\mathbb{E}[\xi f(X+\xi)]-\mathbb{E}X\mathbb{E}f(X)\,,
\end{align*}
so that $\mathbb{E}[Xf(X)]=\lambda\mathbb{E}[\xi f(X+\xi)]$ and we again obtain the characterisation underlying the choice of functional form in \eqref{eq:CPsteinop}. We will exploit this close connection between notions of zero-biasing, which have their roots in Stein's method for Gaussian approximation, and the equations at the heart of Stein's method for compound Poisson approximation in Section \ref{sec:normal}, where we establish central limit theorems with explicit error bounds for sums of associated or negatively associated random variables. 

Also related to this, but away from the compound Poisson setting, is the observation that the zero-biased version $X^\text{z}$ of an infinitely divisible random variable $X$ with zero mean is equal in distribution to $X+Y^\prime$ for some random variable $Y^\prime$ independent of $X$; see the recent work \cite{goldstein24} exploiting and developing this close connection between zero-biasing and infinite divisibility. As we will see in Section \ref{sec:ThirdMom}, this implies that if $X$ also has a vanishing third moment, there is a convex ordering between $X$ and $X^\text{z}$. We use this ordering to establish a Gaussian approximation result for $X$ under the assumption that the first three moments of $X$ match those of the approximating Gaussian distribution. Here, as in Section \ref{sec:normal}, we will use the tools of Stein's method for Gaussian approximation, which we introduce further as needed below.

Finally in this preliminary section, we refer the interested reader to the work of Arras and Houdr\'e \cite{arras19}, who develop the tools of Stein's method for approximation by an infinitely divisible distribution with finite first moment.

\section{Poisson mixtures with infinitely divisible mixing distributions}\label{sec:pois}

In this section our principal aim is to derive bounds on distances between mixed Poisson distributions with infinitely divisible mixing distributions. We exploit the property that such distributions can be size-biased by adding an independent increment, and that we may bound the distance between two such distributions in terms of the proximity of their increments upon size-biasing. This is a first application of our direct approach to exploiting biasing by an independent increment to establish approximations with explicit error bounds. Our main results are presented in Section \ref{subsec:pois_main}, and in Section \ref{subsec:pois_dickman} we give an illustrative application to a Dickman-type limit theorem. The proof of a technical lemma is given in Section \ref{subsec:CPSteinFactors}, following which we conclude this section with some remarks on the case of mixed negative binomial distributions in Section \ref{subsec:negbin}. 

A random variable $Z$ is said to have a mixed Poisson distribution with (non-negative, real-valued) mixing distribution $\xi$, written $Z\sim\text{MP}(\xi)$, if $Z|\xi\sim\mbox{Pois}(\xi)$ has a Poisson distribution with mean $\xi$. That is, if $\mathbb{P}(Z=j)=\frac{\mathbb{E}[e^{-\xi}\xi^j]}{j!}$ for $j=0,1,\ldots$. The family of mixed Poisson distributions is an important one with numerous applications: see \cite{karlis05} for a discussion of some areas of application.

We will focus on the case where the mixing distribution $\xi$ is infinitely divisible, since this will allow us to exploit infinite divisibility of the mixed Poisson random variable $Z$ (see Proposition 8 of \cite{karlis05}). Examples of infinitely divisible distributions on the non-negative real line include the Poisson and compound Poisson distributions, geometric and negative binomial distributions, exponential and gamma distributions, Weibull distributions with shape parameter at most 1, Pareto distributions, \emph{F}-distributions, log-normal distributions, the Dickman distribution, and distributions with a log-convex density function. See Remark 8.12 of \cite{sato99} and Section 11.4 of \cite{arratia19}.   

Since $\xi$ is infinitely divisible, the size-biased version $\xi^\text{s}$ is equal in distribution to $\xi+\eta$, where $\eta$ is a non-negative random variable independent of $\xi$. Then, since $Z^\text{s}\sim\text{MP}(\xi^\text{s})+1$ (see Lemma 2.4 of \cite{arratia19}), we have that $Z^\text{s}$ is equal in distribution to $Z+Y+1$, where $Y\sim\text{MP}(\eta)$ is independent of $Z$. From the definition \eqref{eq:sb_def} of size-biasing we therefore have that, for all $f:\mathbb{Z}^+\to\mathbb{R}$ for which the expectation exists,
\[
\mathbb{E}[Zf(Z)]=\mathbb{E}\xi\sum_{i=1}^\infty\mathbb{P}(Y=i-1)\mathbb{E}f(Z+i)=\mathbb{E}\xi\sum_{i=1}^\infty\frac{1}{(i-1)!}\mathbb{E}[e^{-\eta}\eta^{i-1}]\mathbb{E}f(Z+i)\,.
\]
Following well-established techniques in Stein's method (see, for example, \cite{ross11} for an introduction), this motivates us to define the following Stein equation to assess the proximity of $Z$ to another random variable: For a given function $h:\mathbb{Z}^+\to\mathbb{R}$ we let $f=f_h:\mathbb{Z}^+\to\mathbb{R}$ satisfy $f(0)=0$ and
\begin{equation}\label{eq:MPSteinop}
h(k)-\mathbb{E}h(Z)=\mathbb{E}\xi\sum_{i=1}^\infty\frac{1}{(i-1)!}\mathbb{E}[e^{-\eta}\eta^{i-1}]f(k+i)-kf(k)
\end{equation}
for $k\in\mathbb{Z}^+$. We recognise this as a Stein equation of the form \eqref{eq:CPsteinop} used in compound Poisson approximation problems by Barbour et al$.$ \cite{barbour92_paper}, among many others, with
\begin{equation}\label{eq:lambda_def}
\lambda_i=\frac{\mathbb{E}\xi}{i!}\mathbb{E}[e^{-\eta}\eta^{i-1}]
\end{equation}
for $i\geq1$. We will employ this to assess the proximity of two mixed Poisson distributions by replacing the variable $k$ by another mixed Poisson random variable. Taking expectations, absolute values, and the supremum over $h\in\mathcal{H}$ then gives an equation of the form \eqref{eq:CPstein} that allows us to assess the distance between $Z$ and this second mixed Poisson random variable.

In order to carry out this programme we will need bounds on the solution to \eqref{eq:MPSteinop}.
\begin{lemma} \label{lem:CPSteinFactors}
Let $f$ be the solution to \eqref{eq:MPSteinop} for a given function $h:\mathbb{Z}^+\to\mathbb{R}$, and let $M_0^{(Z)}(\mathcal{H})$ and $M_1^{(Z)}(\mathcal{H})$ be defined analogously to \eqref{eq:SteinFactorDef}. Then
\begin{equation}\label{eq:MPSteinFactor}
M_j^{(Z)}(\mathcal{H}_\text{K})\leq M_j^{(Z)}(\mathcal{H}_\text{TV})\leq\min\left\{1,\frac{1}{(\mathbb{E}\xi)\mathbb{E}e^{-\eta}}\right\}e^\lambda
\end{equation}
for $j=0,1$, where $\lambda=(\mathbb{E}\xi)\mathbb{E}\left[\frac{1}{Y+1}\right]$ and $Y\sim\text{MP}(\eta)$. Furthermore,
\begin{enumerate}
\item[(a)] If $i!\mathbb{E}[e^{-\eta}\eta^{i-1}]\geq(i-1)!\mathbb{E}[e^{\eta}\eta^i]$ for all $i=1,2,\ldots$, then
\begin{align*}
M_0^{(Z)}(\mathcal{H}_\text{K})&\leq\min\left\{1,\sqrt{\frac{2}{e(\mathbb{E}\xi)\mathbb{E}e^{-\eta}}}\right\}\,,\text{ }
M_1^{(Z)}(\mathcal{H}_\text{K})\leq\min\left\{\frac{1}{2},\frac{1}{(\mathbb{E}\xi)\mathbb{E}e^{-\eta}+1}\right\}\,,\\
M_0^{(Z)}(\mathcal{H}_\text{TV})&\leq\min\left\{1,\frac{2}{\sqrt{\beta}}\right\}\,,\text{ and }
M_1^{(Z)}(\mathcal{H}_\text{TV})\leq\min\left\{1,\frac{1}{\beta}\left(\frac{1}{4\beta}+\log^+2\beta\right)\right\}\,,
\end{align*}
where $\beta=(\mathbb{E}\xi)(\mathbb{E}e^{-\eta}-\mathbb{E}[\eta e^{-\eta}])$ and $\log^+$ denotes the positive part of the natural logarithm.
\item[(b)] If $\text{Var}(\xi)<\frac{1}{2}\mathbb{E}\xi$, then
\begin{align*}
M_0^{(Z)}(\mathcal{H}_\text{K})\leq M_0^{(Z)}(\mathcal{H}_\text{TV})\leq\frac{\sqrt{\mathbb{E}\xi}}{\mathbb{E}\xi-2\text{Var}(\xi)}\,,\text{ and }
M_1^{(Z)}(\mathcal{H}_\text{K})\leq M_1^{(Z)}(\mathcal{H}_\text{TV})\leq\frac{1}{\mathbb{E}\xi-2\text{Var}(\xi)}\,.
\end{align*}
\end{enumerate}
\end{lemma}
We defer the proof of Lemma \ref{lem:CPSteinFactors} to Section \ref{subsec:CPSteinFactors} below. The upper bound \eqref{eq:MPSteinFactor} is rather poor for moderate or large $\lambda$, but better bounds are available under the conditions of either part (a) or part (b). The condition in part (b) is typically more straightforward to check, while that in part (a) may be less clear. We note that the condition in part (a) is equivalent to $\mathbb{P}(Y=i)\leq\mathbb{P}(Y=i-1)$ for each $i\geq1$, where $Y\sim\text{MP}(\eta)$. By Corollary 2.7 of \cite{balabdaoui19}, this condition holds if $\eta$ is supported on a subset of $[0,1]$. Examples for which this holds include the following:
\begin{itemize}
\item If $\xi$ has a Poisson distribution then $\xi^\text{s}$ is equal in distribution to $\xi+1$, i.e., $\eta=1$ almost surely, and
\item If $\xi\sim\text{Bin}(n,p)$ has a binomial distribution then $\xi^\text{s}\sim\text{Bin}(n-1,p)+1$, i.e., $\eta$ has a Bernoulli distribution with mean $1-p$.
\end{itemize}
Other examples in which $Y$ has a monotone decreasing mass function include the case where $\xi$ has a gamma distribution (so that $Z$ has a negative binomial distribution, and upon size-biasing we obtain an independent increment $Y$ which has a geometric distribution), and the case where $\xi$ has a Dickman distribution (see Section \ref{subsec:pois_dickman} below for further details).

\subsection{Bounds on distances between mixed Poisson distributions}\label{subsec:pois_main}

In this section we make use of the framework outlined above to establish bounds on distances between mixed Poisson distributions with infinitely divisible mixing distributions. Throughout this section we let $Z\sim\text{MP}(\xi)$ and $W\sim\text{MP}(\mu)$. We are motivated by Lemma 4(b) of Gr\"ubel and Stefanoski \cite{grubel05}, who showed that for any such mixed Poisson random variables
\begin{equation}\label{eq:grubel}
d_\text{W}(W,Z) \leq d_\text{W}(\mu,\xi)\,.
\end{equation}
This bound can be generalised beyond the Wasserstein distance. For some $k\in\mathbb{Z}^+$, we assume that the first $k$ moments of $W$ match those of $Z$ (or, equivalently, the first $k$ moments of $\xi$ match those of $\mu$). Then for a function $h:\mathbb{Z}^+\to\mathbb{R}$ we define $g:\mathbb{R}^+\to\mathbb{R}$ by $g(x)=\mathbb{E}[h(Z)|\xi=x]$, so that
\[
|\mathbb{E}h(W)-\mathbb{E}h(Z)|=|\mathbb{E}g(\mu)-\mathbb{E}g(\xi)|\leq\frac{1}{(k+1)!}\lVert g^{(k+1)}\rVert\inf_{(\mu,\xi)}\mathbb{E}\left[|\mu-\xi|^{k+1}\right]\,,
\]
by Taylor's theorem, where the infimum is taken over all couplings of $\mu$ and $\xi$. Writing $g(x)=\sum_{j=0}^\infty h(j)\frac{e^{-x}x^j}{j!}$ we have that $g^{(k+1)}(x)=\mathbb{E}[\Delta^{k+1}h(Z)|\xi=x]$, where $\Delta$ is the forward difference operator defined by $\Delta h(j)=h(j+1)-h(j)$. Hence,
\[
|\mathbb{E}h(W)-\mathbb{E}h(Z)|\leq\frac{1}{(k+1)!}\lVert\Delta^{k+1}h\rVert\inf_{(\mu,\xi)}\mathbb{E}\left[|\mu-\xi|^{k+1}\right]\,.
\]
Since we may write $d_\text{W}(\mu,\xi)=\inf_{(\mu,\xi)}\mathbb{E}|\mu-\xi|$, \eqref{eq:grubel} follows by letting $k=0$ and taking the supremum over $h\in\mathcal{H}_\text{W}$. Taking the supremum over different classes of test functions $\mathcal{H}$ leads to bounds in different metrics.

In the case where $\mu$ and $\xi$ are infinitely divisible, it is sometimes more useful to have an alternative bound that involves instead the proximity of the independent increments obtained on size-biasing these random variables. We use the rest of this section to give results of this flavour, and give an illustrative example in which they are useful in Section \ref{subsec:pois_dickman} below.    
For the remainder of this section we will assume that $\xi$ and $\mu$ are infinitely divisible, non-negative random variables. We will write $\xi^\text{s}=\xi+\eta$, where $\eta$ is independent of $\xi$, and similarly write $\mu^\text{s}=\mu+\nu$, where $\nu$ is independent of $\mu$. 
\begin{theorem}\label{thm:MP}
Letting $W$ and $Z$ be Poisson mixtures with infinitely divisible mixing distributions as above,
\[
\sup_{h\in\mathcal{H}}|\mathbb{E}h(W)-\mathbb{E}h(Z)|\leq(\mathbb{E}\xi)M_1^{(Z)}(\mathcal{H})d_\text{W}(\nu,\eta)+M_0^{(Z)}(\mathcal{H})|\mathbb{E}\mu-\mathbb{E}\xi|\,.
\]
\end{theorem}
\begin{proof}
For $h\in\mathcal{H}$ we let $f$ denote the solution to \eqref{eq:MPSteinop} with $f(0)=0$. Letting $Y\sim\text{MP}(\eta)$ be independent of $W$, we then have that
\begin{align*}
\sup_{h\in\mathcal{H}}|\mathbb{E}h(W)-\mathbb{E}h(Z)| &=\sup_{h\in\mathcal{H}}\left|(\mathbb{E}\xi)\mathbb{E}\sum_{i=1}^\infty\frac{1}{(i-1)!}\mathbb{E}[e^{-\eta}\eta^{i-1}]f(W+i)-\mathbb{E}[Wf(W)]\right|\\
&=\sup_{h\in\mathcal{H}}\left|(\mathbb{E}\xi)\mathbb{E}f(W+Y+1)-(\mathbb{E}\mu)\mathbb{E}f(W^\text{s})\right|\\
&\leq(\mathbb{E}\xi)\sup_{h\in\mathcal{H}}\left|\mathbb{E}f(W+Y+1)-\mathbb{E}f(W^\text{s})\right|
+\sup_{h\in\mathcal{H}}\left|\mathbb{E}f(W^\text{s})\right|\left|\mathbb{E}\mu-\mathbb{E}\xi\right|\,.
\end{align*}
We construct $W^\text{s}$ as $W+V+1$, where $V\sim\text{MP}(\nu)$ is independent of $W$, so that
\begin{align*}
\sup_{h\in\mathcal{H}}|\mathbb{E}h(W)-\mathbb{E}h(Z)|&\leq(\mathbb{E}\xi)\sup_{h\in\mathcal{H}}\left|\mathbb{E}f(W+Y+1)-\mathbb{E}f(W+V+1)\right|+\sup_{h\in\mathcal{H}}\left|\mathbb{E}f(W^\text{s})\right|\left|\mathbb{E}\mu-\mathbb{E}\xi\right|\\
&\leq(\mathbb{E}\xi)M_1^{(Z)}(\mathcal{H})d_\text{W}(V,Y)+M_0^{(Z)}(\mathcal{H})|\mathbb{E}\mu-\mathbb{E}\xi|\\
&\leq(\mathbb{E}\xi)M_1^{(Z)}(\mathcal{H})d_\text{W}(\nu,\eta)+M_0^{(Z)}(\mathcal{H})|\mathbb{E}\mu-\mathbb{E}\xi|\,,
\end{align*}
where the final equality follows from \eqref{eq:grubel}.
\end{proof}
As a corollary we obtain a generalisation of Theorem 1.C(ii) of \cite{barbour92_book} under the additional condition of infinite divisibility, for which we need a little additional notation. From equation (88) of \cite{arratia19}, the characteristic function of $Z$ may be written as
\[
\mathbb{E}e^{iuZ}=\exp\left\{\mathbb{E}\xi\left(ium_\xi(\{0\})+\int_0^\infty\frac{e^{iuy}-1}{y}m_\xi(\text{d}y)\right)\right\}\,,
\]
where $m_\xi$ is the L\'evy measure associated with $\xi$, and is the distribution of the increment $\eta$ obtained on size-biasing; see Section 11.2 of \cite{arratia19}. We will write $m_\xi\geq_\text{st}m_\mu$ to denote that the L\'evy measure of $\xi$ is stochastically larger than that of $\mu$.
\begin{corollary}\label{cor:MP_ordering}
Let $W$ and $Z$ be as in Theorem \ref{thm:MP}, and assume that either $m_\xi\geq_\text{st}m_\mu$ or $m_\mu\geq_\text{st}m_\xi$. Then
\[
\sup_{h\in\mathcal{H}}|\mathbb{E}h(W)-\mathbb{E}h(Z)|\leq(\mathbb{E}\xi)M_1^{(Z)}(\mathcal{H})
\left|\frac{\text{Var}(\mu)}{\mathbb{E}\mu}-\frac{\text{Var}(\xi)}{\mathbb{E}\xi}\right|+M_0^{(Z)}(\mathcal{H})|\mathbb{E}\mu-\mathbb{E}\xi|\,.
\]
\end{corollary}
\begin{proof}
We complete the proof under the assumption that $m_\xi\geq_\text{st}m_\mu$; the argument is analogous if the reverse ordering holds. This implies that $\eta$ is stochastically larger than $\nu$, and hence that
\[
d_\text{W}(\nu,\eta)=\mathbb{E}\eta-\mathbb{E}\nu=\mathbb{E}\xi^\text{s}-\mathbb{E}\xi-\mathbb{E}\mu^\text{s}+\mathbb{E}\mu=\frac{\text{Var}(\xi)}{\mathbb{E}\xi}-\frac{\text{Var}(\mu)}{\mathbb{E}\mu}\,,
\]
where the final equality uses the definition \eqref{eq:sb_def}. The result then follows from Theorem \ref{thm:MP}.
\end{proof}
We note that if $Z$ has a Poisson distribution then the associated L\'evy measure is a point mass at zero, so Corollary \ref{cor:MP_ordering} yields the bound of Theorem 1.C(ii) of \cite{barbour92_book} in the special case where $Z$ is Poisson with mean equal to $\mathbb{E}\mu$ and $\mathcal{H}=\mathcal{H}_\text{TV}$, though we note that we additionally require $W$ to have infinitely divisible mixing distribution. 

Bounds on the terms of the form $M_j^{(Z)}(\mathcal{H})$ appearing in Theorem \ref{thm:MP} and Corollary \ref{cor:MP_ordering} are given in Lemma \ref{lem:CPSteinFactors}; in applying these results the mixed Poisson distribution playing the role of $Z$ should be chosen as the one which gives such bounds of the best order. 

\subsection{Example: A mixed Poisson--Dickman approximation}\label{subsec:pois_dickman}

To demonstrate Theorem \ref{thm:MP} we consider here an illustrative example. For $n=1,2,\ldots$ and some $c>0$, we let
\begin{equation}\label{eq:MPD_Wdef}
W=\sum_{k=1}^nB_kP_k\,,
\end{equation}
where $B_1,B_2,\ldots$ is a sequence of independent Bernoulli random variables with $\mathbb{E}B_k=1/k$, and $P_1,P_2,\ldots$ is a sequence of independent Poisson random variables, also independent of the $B_k$, with $P_k\sim\text{Pois}(ck/n)$. 

We may write $W\sim\text{MP}(c\xi)$, where $\xi=\frac{1}{n}\sum_{k=1}^nkB_k$. This $\xi$ is a weighted sum of Bernoulli random variables of the form studied by Bhattacharjee and Goldstein \cite{bhattacharjee19} and Bhattacharjee and Schulte \cite{bhattacharjee22}, motivated by applications to the Quickselect algorithm and to records processes, who provide explicit bounds in the approximation of such a $\xi$ by a Dickman distribution in the Wasserstein-2 and Kolmogorov distances, respectively. Dickman approximations of more general weighted sums of the form $\sum_{k=1}^nB_kY_k$ for some independent random variables $Y_1,Y_2,\ldots$ whose means grow linearly with $k$ are also studied in \cite{bhattacharjee19,bhattacharjee22}, with results that perform particularly well when the variances of the $Y_k$ are small. Our present example is motivated by an alternative limit which may arise when this is not the case. 

Throughout this section we let $D\sim\text{Dickman}(1)$ have a standard Dickman distribution with $\mathbb{E}D=1$ and density function $e^{-\gamma}\rho(x)$ for $x>0$, where $\rho$ is Dickman's function and $\gamma=0.5772\ldots$ is Euler's constant. Given the Dickman approximation results for $\xi$ discussed above, it would be reasonable to approximate $W$ by $Z\sim\text{MP}(cD)$. We do this in the total variation distance for concreteness. A natural first approach to this would be to write
\[
d_\text{TV}(W,Z)\leq d_\text{W}(W,Z)\leq cd_\text{W}(\xi,D)\,,
\]
where the first inequality uses the comparability of $d_\text{TV}$ and $d_\text{W}$ on the integers, and the second uses both \eqref{eq:grubel} and scaling properties of Wasserstein distance. Unfortunately, bounds on $d_\text{W}(\xi,D)$ do not seem to be readily available in the literature, so we use a slightly less direct approach via Theorem \ref{thm:MP}. Demonstrating this approach is the main focus of this section.

From Example 11.12 of \cite{arratia19}, we have that $D^\text{s}$ is equal in distribution to $D+U$, where $U\sim\text{Unif}(0,1)$ has a uniform distribution and is independent of $D$. We may then construct
$(cD)^\text{s}=cD^\text{s}=cD+V$, where $V\sim\text{Unif}(0,c)$ is independent of $D$. We may therefore construct $Z^\text{s}$ as $Z+Y+1$, where $Y\sim\text{MP}(V)$ is independent of $Z$. We note that 
\[
\mathbb{P}(Y=j)=\frac{1}{cj!}\int_0^ce^{-v}v^j\,\text{d}v=\frac{1}{c}\mathbb{P}(X>j)\,,
\]
where $X\sim\text{Pois}(c)$, so that the mass function of $Y$ is monotonically decreasing. For a given function $h:\mathbb{Z}^+\to\mathbb{R}$, letting $f$ denote the solution to
\[
h(k)-\mathbb{E}h(Z)=c\sum_{i=1}^\infty\frac{1}{(i-1)!}\mathbb{E}[e^{-V}V^{i-1}]f(k+i)-kf(k)
\]
with $f(0)=0$, Lemma \ref{lem:CPSteinFactors}(a) thus gives us that
\[
M_1^{(Z)}(\mathcal{H}_\text{TV})\leq\min\left\{1,\frac{1}{\beta}\left(\frac{1}{4\beta}+\log^+2\beta\right)\right\}\,,
\]
where $\beta=c(\mathbb{E}e^{-V}-\mathbb{E}[Ve^{-V}])=ce^{-c}$.

To construct $W^\text{s}$ we use well-known rules for size-biasing a sum of independent random variables (see Section 2.4 of \cite{arratia19}). Let $I$ be uniformly distributed on $\{1,\ldots,n\}$, independent of all else, and construct $W^\text{s}=W-B_IP_I+(B_IP_I)^\text{s}$. A straightforward calculation confirms that $(B_IP_I)^\text{s}$ is equal in distribution to $P_I+1$, so that $W^\text{s}$ is equal in distribution to $W+(1-B_I)P_I+1$, where $(1-B_I)P_I\sim\text{MP}((1-B_I)\frac{cI}{n})$.

From Theorem \ref{thm:MP} we thus have
\begin{equation}\label{eq:MPD_TV}
d_\text{TV}(W,Z)\leq cM_1^{(Z)}(\mathcal{H}_\text{TV})d_W\left((1-B_I)\frac{cI}{n},V\right)= c^2M_1^{(Z)}(\mathcal{H}_\text{TV})d_W\left((1-B_I)\frac{I}{n},U\right)\,.
\end{equation}
Furthermore,
\[
d_W\left((1-B_I)\frac{I}{n},U\right)=\int_0^1\left|\mathbb{P}\left((1-B_I)\frac{I}{n}\leq x\right)-x\right|\,\text{d}x\,,
\]
and
\begin{align*}
\mathbb{P}\left((1-B_I)\frac{I}{n}\leq x\right)&=\frac{1}{n}\sum_{k=1}^n\mathbb{P}\left((1-B_k)\frac{k}{n}\leq x\right)=\frac{1}{n}\sum_{k=1}^n\mathbb{P}\left(B_k\geq1-\frac{nx}{k}\right)\\
&=\frac{1}{n}\sum_{k=1}^{\lfloor nx\rfloor}1+\frac{1}{n}\sum_{k=\lfloor nx\rfloor+1}^n\frac{1}{k}
=\frac{\lfloor nx\rfloor}{n}+\frac{1}{n}\sum_{k=\lfloor nx\rfloor+1}^n\frac{1}{k}\,,
\end{align*}
where $\lfloor\cdot\rfloor$ is the floor function. Hence, $d_W\left((1-B_I)\frac{I}{n},U\right)\leq a+b$, where
\begin{align*}
a&=\int_0^1\left|\frac{\lfloor nx\rfloor}{n}-x\right|\,\text{d}x=\int_0^1\left(x-\frac{\lfloor nx\rfloor}{n}\right)\,\text{d}x=\frac{1}{2}-\frac{n-1}{2n}=\frac{1}{2n}\,,\\
b&=\frac{1}{n}\int_0^1\sum_{k=\lfloor nx\rfloor+1}^n\frac{1}{k}\,\text{d}x=\frac{1}{n}\int_0^1\left(H_n-H_{\lfloor nx\rfloor}\right)\,\text{d}x\,,
\end{align*}
and $H_n=\sum_{k=1}^n1/k$ is the $n$th harmonic number. Since $\log(n+1)<H_n<\log(n)+1$ for each $n$, we have
\[
H_n-H_{\lfloor nx\rfloor}\leq1+\log\left(\frac{n}{\lfloor nx\rfloor+1}\right)\leq1-\log(x)\,,
\]
and $b\leq\frac{1}{n}\int_0^1(1-\log x)\,\text{d}x=\frac{2}{n}$. Hence, $d_W\left((1-B_I)\frac{I}{n},U\right)\leq\frac{5}{2n}$ and from \eqref{eq:MPD_TV} we obtain the following bound.
\begin{proposition}
Let $W$ be as in \eqref{eq:MPD_Wdef} and $Z\sim\text{MP}(cD)$. Then
\[
d_\text{TV}(W,Z)\leq\frac{5c^2}{2n}\min\left\{1,\frac{1}{\beta}\left(\frac{1}{4\beta}+\log^+2\beta\right)\right\}\,,
\]
where $\beta=ce^{-c}$.
\end{proposition}
An approximation result in Kolmogorov distance can be derived similarly, replacing $M_1^{(Z)}(\mathcal{H}_\text{TV})$ with $M_1^{(Z)}(\mathcal{H}_\text{K})$ and using the bound on this latter quantity from Lemma \ref{lem:CPSteinFactors}(a). 

\subsection{Proof of Lemma \ref{lem:CPSteinFactors}}\label{subsec:CPSteinFactors}

We use this section to give the proof of Lemma \ref{lem:CPSteinFactors}. We firstly note that $M_j^{(Z)}(\mathcal{H}_\text{K})\leq M_j^{(Z)}(\mathcal{H}_\text{TV})$ for $j=0,1$ since $\mathcal{H}_\text{K}\subseteq\mathcal{H}_\text{TV}$. The upper bound \eqref{eq:MPSteinFactor} then follows immediately from Theorem 4 of \cite{barbour92_paper}, noting that with $\lambda_i$ given by \eqref{eq:lambda_def} we have
\[
\sum_{i=1}^\infty\lambda_i=\mathbb{E}\xi\sum_{i=0}^\infty\frac{1}{(i+1)!}\mathbb{E}[e^{-\eta}\eta^{i}]=\mathbb{E}\xi\sum_{i=0}^\infty\frac{1}{i+1}\mathbb{P}(Y=i)=(\mathbb{E}\xi)\mathbb{E}\left[\frac{1}{Y+1}\right]\,.
\]

For part (a), we note that the assumption is equivalent to $i\lambda_i\geq(i+1)\lambda_{i+1}$ for all $i\geq1$. The bounds on $M_0^{(Z)}(\mathcal{H}_\text{K})$ and $M_1^{(Z)}(\mathcal{H}_\text{K})$ then follow from Proposition 1.1 of \cite{barbour00}. The bounds on $M_0^{(Z)}(\mathcal{H}_\text{TV})$ and $M_1^{(Z)}(\mathcal{H}_\text{TV})$ follow from Proposition 2 and Theorem 5, respectively, of \cite{barbour92_paper}.

For part (b) we use the results of Barbour and Xia \cite{barbour99}, who give bounds on the quantities $M_j^{(Z)}(\mathcal{H}_\text{TV})$ for $j=0,1$ under the assumption \eqref{eq:bx99}. In our setting we have that 
\begin{align*}
\frac{\sum_{j=1}^\infty j(j-1)\lambda_j}{\sum_{j=1}^\infty j\lambda_j}&=\frac{\sum_{j=1}^\infty\frac{1}{(j-1)!}\mathbb{E}[e^{-\eta}\eta^j]}{\sum_{j=1}^\infty\frac{1}{(j-1)!}\mathbb{E}[e^{-\eta}\eta^{j-1}]}=\frac{\mathbb{E}[e^{-\eta}\sum_{j=1}^\infty\frac{\eta^j}{(j-1)!}]}{\mathbb{E}[e^{-\eta}\sum_{j=1}^\infty\frac{\eta^{j-1}}{(j-1)!}]}\\
&=\mathbb{E}\eta=\mathbb{E}[\xi^\text{s}]-\mathbb{E}\xi=\frac{\mathbb{E}[\xi^2]}{\mathbb{E}\xi}-\mathbb{E}\xi=\frac{\text{Var}(\xi)}{\mathbb{E}\xi}\,,
\end{align*}
so that \eqref{eq:bx99} holds if $\text{Var}(\xi)<\frac{1}{2}\mathbb{E}\xi$. The bounds of part (b) of our lemma then follow immediately from Theorem 2.5 of \cite{barbour99}.  

Finally, we remark that there are bounds on $M_j^{(Z)}(\mathcal{H}_\text{K})$ for $j=0,1$ available under slightly more relaxed conditions than those imposed in parts (a) and (b) of our lemma: see Theorem 3.1 of \cite{barbour98} for a relaxation of the monotonicity condition of part (a), and Theorem 1.1 of \cite{daly17} for a relaxation of the condition in part (b). We do not give further details here to keep our exposition simpler, and since these relaxations do not provide significant benefit in the examples and applications we consider, in which we focus more on bounding the stronger total variation distance. 

\subsection{The mixed negative binomial case}\label{subsec:negbin}

The techniques of this section may also be applied in other settings, though with additional technical challenges in some cases. Beyond the mixed Poisson case, it is natural to consider a mixed negative binomial random variable $Z$ with mass function
\[
\mathbb{P}(Z=j)=\mathbb{E}\left[\frac{\Gamma(j+\xi)}{j!\Gamma(\xi)}(1-p)^jp^\xi\right]\,,
\]
for $j\in\mathbb{Z}^+$, some $p\in[0,1]$ and a non-negative random variable $\xi$, where $\Gamma(\cdot)$ is the gamma function. We denote this by $Z\sim\text{MNB}(\xi,p)$ As before, we will assume that the mixing distribution is infinitely divisible, and write $\xi^\text{s}=\xi+\eta$. In this case, a straightforward calculation using results from Section 2.2 of \cite{arratia19} gives us that we may construct $Z^\text{s}$ as $Z+Y+G+1$, where these three random variables are independent, $Y\sim\text{MNB}(\eta)$ and $G$ has a geometric distribution with mass function $\mathbb{P}(G=j)=p(1-p)^j$ for $j\in\mathbb{Z}^+$. Writing
\begin{align*}
\mathbb{E}[Zf(Z)]&=\frac{(1-p)\mathbb{E}\xi}{p}\sum_{i=1}^\infty\mathbb{P}(Y+G=i-1)\mathbb{E}f(Z+i)\\
&=\mathbb{E}\xi\sum_{i=1}^\infty\sum_{j=0}^{i-1}(1-p)^{j+1}\mathbb{P}(Y=i-j-1)\mathbb{E}f(Z+i)\\
&=\mathbb{E}\xi\sum_{i=1}^\infty\sum_{j=0}^{i-1}\frac{(1-p)^{i}}{(i-j-1)!}\mathbb{E}\left[\frac{\Gamma(i-j-1+\eta)}{\Gamma(\eta)}p^\eta\right]\mathbb{E}f(Z+i)\,,
\end{align*}
the equivalent of \eqref{eq:MPSteinop} in this setting is
\[
h(k)-\mathbb{E}h(Z)=\mathbb{E}\xi\sum_{i=1}^\infty\sum_{j=0}^{i-1}\frac{(1-p)^{i}}{(i-j-1)!}\mathbb{E}\left[\frac{\Gamma(i-j-1+\eta)}{\Gamma(\eta)}p^\eta\right]f(k+i)-kf(k)\,.
\]
Then, for example, to be able to use bounds analogous to those in Lemma \ref{lem:CPSteinFactors}(a), which we recall are typically of a better order than those available without the monotonicity conditions assumed therein, we would need to assume that the inner sum $\sum_{j=0}^{i-1}\frac{(1-p)^{i}}{(i-j-1)!}\mathbb{E}\left[\frac{\Gamma(i-j-1+\eta)}{\Gamma(\eta)}p^\eta\right]$ is monotonically decreasing in $i=1,2,\ldots$. Finding examples or sufficient conditions for this seems to be more challenging than in the corresponding Poisson case. To consider just one example, if $\xi$ is Poisson then $\eta=1$ almost surely and we would need $i(1-p)^i$ to be monotonically decreasing in $i$, which occurs only if $p\geq1/2$. We leave further exploration of this question, and that of other extensions of these techniques, for future research.

\section{Gaussian approximation for sums of associated or negatively associated random variables}\label{sec:normal}

In this section we derive error bounds for Gaussian approximation in the Wasserstein distance of a sum of associated or negatively associated random variables. We will again exploit the property of biasing with an independent increment, though our approach here is less direct than in Section \ref{sec:pois}. We begin with a general Gaussian approximation result (Theorem \ref{thm:Gauss} below) which makes use of the fact that we may non-zero-bias a compound Poisson random variable by adding an independent increment, as discussed in Section \ref{sec:cp}, in conjunction with Stein's method for Gaussian approximation, for which we refer the reader to \cite{chen11} for an introduction, and the compound Poisson approximation techniques already introduced. 

We then go on to apply our Theorem \ref{thm:Gauss} to a sum of associated or negatively associated random variables in Section \ref{subsec:normal_ass}, making use of compound Poisson approximation results from \cite{daly13}. Applications to simple random sampling and urn models with overflow are considered in Sections \ref{subsec:normal_sample} and \ref{subsec:overflow}, respectively.
\begin{theorem}\label{thm:Gauss}
Let $W$ be a real-valued random variable with $\mathbb{E}W=\theta$ and $\text{Var}(W)=\sigma^2$, and $Z\sim\text{N}(\theta,\sigma^2)$ be a Gaussian random variable whose first two moments match those of $W$. Let $\xi$ be a real-valued random variable, independent of $W$, with $\mathbb{E}\xi=\theta/\lambda$ for some $\lambda>0$. Then
\[
d_\text{W}(W,Z)\leq\frac{\mathbb{E}|\xi|^3}{\mathbb{E}[\xi^2]}+\sqrt{\frac{2}{\pi}}\sigma^{-1}|\sigma^2-\lambda\mathbb{E}[\xi^2]|+\sup_{h\in\mathcal{H}_{\text{W}}}\left|\lambda\mathbb{E}[\xi f(W+\xi)]-\mathbb{E}[Wf(W)]\right|\,,
\] 
where $f:\mathbb{R}\to\mathbb{R}$ is the solution of
\begin{equation}\label{eq:GaussStein}
h(x)-\mathbb{E}h(Z)=\sigma^2f^\prime(x)-(x-\theta)f(x)
\end{equation}
for $h\in\mathcal{H}_\text{W}$. 
\end{theorem} 
\begin{proof}
Letting $f$ be as in \eqref{eq:GaussStein} we follow Stein's method for Gaussian approximation (see \cite{chen11}) to write
\begin{equation}\label{eq:GaussPf1}
d_\text{W}(W,Z)\leq\sup_{h\in\mathcal{H}_\text{W}}|\mathbb{E}[\sigma^2f^\prime(W)-(W-\theta)f(W)]|\,,
\end{equation}
noting that this is usually written in the setting with $\theta=0$ and $\sigma^2=1$. The setting we use here, with a general mean and variance, is a straightforward extension that we use for convenience in the work that follows. Following the arguments of the proof of Lemma 2.4 of \cite{chen11}, we have that, for $h\in\mathcal{H}_\text{W}$,
\begin{equation}\label{eq:GaussSteinFactors}
\sup_{x\in\mathbb{R}}|f^\prime(x)|\leq\sqrt{\frac{2}{\pi}}\sigma^{-1}\,,\text{ and }\sup_{x\in\mathbb{R}}|f^{\prime\prime}(x)|\leq2\sigma^{-2}\,.
\end{equation}
Now, using \eqref{eq:nz_def} we write
\begin{align}\label{eq:GaussPf2}
\nonumber\mathbb{E}[\sigma^2f^\prime(W)-(W-\theta)f(W)]
&=\sigma^2\left[\mathbb{E}f^\prime(W)-\mathbb{E}f^\prime(W^\text{nz})\right]\\
&=\sigma^2\left[\mathbb{E}f^\prime(W)-\mathbb{E}f^\prime(W+\xi^\text{gz})+\mathbb{E}f^\prime(W+\xi^\text{gz})-\mathbb{E}f^\prime(W^\text{nz})\right]\,.
\end{align}
From \eqref{eq:GaussSteinFactors} we have that, for $h\in\mathcal{H}_\text{W}$,
\begin{equation}\label{eq:GaussPf3}
\left|\mathbb{E}f^\prime(W)-\mathbb{E}f^\prime(W+\xi^\text{gz})\right|
\leq2\sigma^{-2}\mathbb{E}|\xi^\text{gz}|=\frac{\mathbb{E}|\xi|^3}{\sigma^2\mathbb{E}[\xi^2]}\,,
\end{equation}
where the final equality uses the definition \eqref{eq:gz_def}. Again using the definitions \eqref{eq:nz_def} and \eqref{eq:gz_def}, together with the independence of $W$ and $\xi$, we have
\begin{align}\label{eq:GaussPf4}
\nonumber\mathbb{E}f^\prime(W+\xi^\text{gz})-\mathbb{E}f^\prime(W^\text{nz})
&=\frac{\mathbb{E}[\xi(f(W+\xi)-f(W))]}{\mathbb{E}[\xi^2]}-\frac{\mathbb{E}[(W-\theta)f(W)]}{\sigma^2}\\
\nonumber&=\left(\frac{1}{\mathbb{E}[\xi^2]}-\frac{\lambda}{\sigma^2}\right)\mathbb{E}[\xi f(W+\xi)]-\theta\left(\frac{1}{\lambda\mathbb{E}[\xi^2]}-\frac{1}{\sigma^2}\right)\mathbb{E}f(W)\\
\nonumber&\qquad+\frac{1}{\sigma^2}\left(\lambda\mathbb{E}[\xi f(W+\xi)]-\mathbb{E}[Wf(W)]\right)\\
\nonumber&=\left(\frac{1}{\mathbb{E}[\xi^2]}-\frac{\lambda}{\sigma^2}\right)\left(\mathbb{E}[\xi f(W+\xi)]-\mathbb{E}[\xi f(W)]\right)\\
&\qquad+\frac{1}{\sigma^2}\left(\lambda\mathbb{E}[\xi f(W+\xi)]-\mathbb{E}[Wf(W)]\right)\,.
\end{align}
By Taylor's theorem we write
\[
\left(\frac{1}{\mathbb{E}[\xi^2]}-\frac{\lambda}{\sigma^2}\right)\left(\mathbb{E}[\xi f(W+\xi)]-\mathbb{E}[\xi f(W)]\right)=\left(\frac{1}{\mathbb{E}[\xi^2]}-\frac{\lambda}{\sigma^2}\right)\mathbb{E}[\xi^2f^\prime(\kappa)]
\]
for some $\kappa$ between $W$ and $W+\xi$, so that
\begin{equation}\label{eq:GaussPf5}
\left|\frac{1}{\mathbb{E}[\xi^2]}-\frac{\lambda}{\sigma^2}\right|\left|\mathbb{E}[\xi f(W+\xi)]-\mathbb{E}[\xi f(W)]\right|\leq\sqrt{\frac{2}{\pi}}\sigma^{-1}\left|1-\frac{\lambda\mathbb{E}[\xi^2]}{\sigma^2}\right|
\end{equation}
by \eqref{eq:GaussSteinFactors}. The proof is completed by combining \eqref{eq:GaussPf1} with \eqref{eq:GaussPf2}--\eqref{eq:GaussPf5}.
\end{proof}

\subsection{Gaussian approximation under association or negative 
association}\label{subsec:normal_ass}

We now apply Theorem \ref{thm:Gauss} to the case where $W=Y_1+\cdots+Y_n$ is a sum of non-negative, integer-valued random variables which are either associated or negatively associated. Recall that $Y_1,\ldots,Y_n$ are said to be associated if 
\[
\mathbb{E}[g(Y_i,1\leq i\leq n)h(Y_i,1\leq i\leq n)]\geq\mathbb{E}[g(Y_i,1\leq i\leq n)]\mathbb{E}[h(Y_i,1\leq i\leq n)]
\]  
for all non-decreasing functions $g$ and $h$. Similarly $Y_1,\ldots,Y_n$ are said to be negatively associated if
\[
\mathbb{E}[g(Y_i,i\in\mathcal{I})h(Y_i,i\in\mathcal{J})]\leq\mathbb{E}[g(Y_i,i\in\mathcal{I})]\mathbb{E}[h(Y_i,i\in\mathcal{J})]
\]
for all non-decreasing functions $g$ and $h$, and all disjoint sets $\mathcal{I},\mathcal{J}\subseteq\{1,\ldots,n\}$. These two definitions were introduced by \cite{esary67} and \cite{joagdev83}, respectively, and have since found numerous applications. Our interest is in establishing central limit theorems with explicit error bounds in Wasserstein distance in this setting, along the lines of the univariate results presented by Goldstein and Wiroonsri \cite{goldstein18} and Wiroonsri \cite{wiroonsri18}, but without the boundedness assumptions on the underlying random variables required by these latter results.

We obtain the following result in the negatively associated setting, which we state in terms of the standardised random variables for easier comparison with results of \cite{wiroonsri18}.
\begin{theorem}\label{thm:negass}
Let $Y_1,\ldots,Y_n$ be non-negative, integer-valued random variables which are negatively associated, and let $W=Y_1+\cdots+Y_n$ with $\sigma^2=\text{Var}(W)$. Then
\[
d_\text{W}(\widetilde{W},\widetilde{Z})\leq \frac{\sum_{k=1}^n\mathbb{E}[Y_k^3]}{\sigma\sum_{k=1}^n\mathbb{E}[Y_k^2]}
+\sqrt{\frac{8}{\pi}}\sigma^{-2}\left(-\sum_{i\not=j}\text{Cov}(Y_i,Y_j)+\sum_{k=1}^n\mathbb{E}[Y_k]^2\right)\,,
\]  
where $\widetilde{W}=\sigma^{-1}(W-\mathbb{E}W)$ and $\widetilde{Z}\sim\text{N}(0,1)$.
\end{theorem}  
\begin{proof}
We combine Theorem \ref{thm:Gauss} with compound Poisson approximation results of Daly \cite{daly13}, who showed in the proof of his Theorem 1.1 that if we choose $\mathbb{P}(\xi=j)=\lambda^{-1}\sum_{k=1}^n\mathbb{P}(Y_k=j)$ in our setting then we obtain
\begin{align*}
\sup_{h\in\mathcal{H}_{\text{W}}}\left|\lambda\mathbb{E}[\xi f(W+\xi)]-\mathbb{E}[Wf(W)]\right|&\leq\sup_{h\in\mathcal{H}_{\text{W}}}\lVert\Delta f\rVert\left(\sum_{k=1}^n\mathbb{E}[Y_k^2]-\sigma^2\right)\\
&\leq\sqrt{\frac{2}{\pi}}\sigma^{-1}\left(-\sum_{i\not=j}\text{Cov}(Y_i,Y_j)+\sum_{k=1}^n\mathbb{E}[Y_k]^2\right)\,,
\end{align*}
where the final inequality follows from \eqref{eq:GaussSteinFactors} and we note that the negative association property implies that these covariances are non-positive.

With this choice of $\xi$ we have that $\mathbb{E}\xi=\lambda^{-1}\mathbb{E}[W]$, $\mathbb{E}[\xi^2]=\lambda^{-1}\sum_{k=1}^n\mathbb{E}[Y_k^2]$, and $\mathbb{E}|\xi^3|=\mathbb{E}[\xi^3]=\lambda^{-1}\sum_{k=1}^n\mathbb{E}[Y_k^3]$. We also note that
\[
\sigma^2-\lambda\mathbb{E}[\xi^2]=\sum_{i\not=j}\text{Cov}(Y_i,Y_j)-\sum_{k=1}^n\mathbb{E}[Y_k]^2\,.
\]
Theorem \ref{thm:Gauss} then gives
\[
d_\text{W}(W,Z)\leq \frac{\sum_{k=1}^n\mathbb{E}[Y_k^3]}{\sum_{k=1}^n\mathbb{E}[Y_k^2]}
+\sqrt{\frac{8}{\pi}}\sigma^{-1}\left(-\sum_{i\not=j}\text{Cov}(Y_i,Y_j)+\sum_{k=1}^n\mathbb{E}[Y_k]^2\right)\,,
\]
where $Z\sim\text{N}(\mathbb{E}W,\sigma^2)$. Scaling properties of Wasserstein distance give $d_\text{W}(\widetilde{W},\widetilde{Z})=\sigma^{-1}d_\text{W}(W,Z)$, completing the proof.
\end{proof}
This gives a result comparable to Theorem 1.3 of Wiroonsri \cite{wiroonsri18}, who showed that if $\widetilde{W}=\widetilde{Y}_1+\cdots+\widetilde{Y}_n$ is a sum of mean-zero, negatively associated random variables with $|\widetilde{Y}_i|\leq B$ for each $i$ and with $\text{Var}(\widetilde{W})=1$, then
\begin{equation}\label{eq:Wnegass}
d_\text{W}(\widetilde{W},\widetilde{Z})\leq5B-5.2\sum_{i\not=j}\text{Cov}(\widetilde{Y_i},\widetilde{Y}_j)\,.
\end{equation}
Note that this does not require the $\widetilde{Y}_i$ be non-negative and integer-valued; our Theorem \ref{thm:negass} has this assumption in order to be able to employ the compound Poisson approximation framework from \cite{daly13}. However, our result does remove the boundedness needed by \eqref{eq:Wnegass}. In cases where this boundedness does hold, the term proportional to $\sigma^{-2}\sum_{k=1}^n\mathbb{E}[Y_k]^2$ in the upper bound of Theorem \ref{thm:negass} means that our result may perform significantly worse than \eqref{eq:Wnegass}. For example, if the $Y_k$ are independent Bernoulli random variables with fixed mean $p$ then \eqref{eq:Wnegass} gives an upper bound in the approximation of their (standardised) sum of order $O(n^{-1/2})$, while $\sigma^{-2}\sum_{k=1}^n\mathbb{E}[Y_k]^2=p/(1-p)$. This latter term is better behaved in other situations. For example, if the $Y_k$ are independent Bernoulli random variables, each with mean $\lambda\log(n)/n$ for some $\lambda>0$ then this term is of order $O(\log(n)/n)$. In this example, both our Theorem \ref{thm:negass} and \eqref{eq:Wnegass} give upper bounds of order $O(1/\log(n))$, with a slightly better constant in the leading term in Theorem \ref{thm:negass}. We will give further illustrations of our Theorem \ref{thm:negass} in Sections \ref{subsec:normal_sample} and \ref{subsec:overflow} below, in the settings of simple random sampling and urn models with overflow, respectively, in which our result may significantly outperform \eqref{eq:Wnegass}.   

Our Theorem \ref{thm:negass} makes use of a known compound Poisson approximation result for sums of negatively associated random variables to derive a Gaussian approximation result. It would, of course, be simpler to derive such a bound directly with the triangle inequality in conjunction with the same compound Poisson approximation bound and a bound on the proximity of the compound Poisson distribution to Gaussian. However, this would lead to terms involving bounds on the solution to the compound Poisson equation \eqref{eq:CPsteinop} which, as discussed in Section \ref{sec:cp}, have a poor dependence on parameters of the problem in general. The approach we use in our Theorem \ref{thm:negass} lets us bypass this difficulty by using the solution to the Gaussian equation \eqref{eq:GaussStein} instead of having to deal with the solution to \eqref{eq:CPsteinop}.

We obtain an analogous result in the case of a sum of associated random variables. 
\begin{theorem}\label{thm:ass}
Let $Y_1,\ldots,Y_n$ be non-negative, integer-valued random variables which are associated, and let $W=Y_1+\cdots+Y_n$ with $\sigma^2=\text{Var}(W)$. Then
\[
d_\text{W}(\widetilde{W},\widetilde{Z})\leq\frac{\sum_{k=1}^n\mathbb{E}[Y_k^3]}{\sigma\sum_{k=1}^n\mathbb{E}[Y_k^2]}
+\sqrt{\frac{8}{\pi}}\sigma^{-2}\left(\sum_{i\not=j}\text{Cov}(Y_i,Y_j)+\sum_{k=1}^n\mathbb{E}[Y_k]^2\right)\,,
\]  
where $\widetilde{W}=\sigma^{-1}(W-\mathbb{E}W)$ and $\widetilde{Z}\sim\text{N}(0,1)$.
\end{theorem} 
\begin{proof}
We proceed similarly to the proof of Theorem \ref{thm:negass}, making the same choice of $\xi$, with $\mathbb{P}(\xi=j)=\lambda^{-1}\sum_{k=1}^n\mathbb{P}(Y_k=j)$. In this setting, the proof of Theorem 1.2 of Daly \cite{daly13} gives
\begin{align*}
\sup_{h\in\mathcal{H}_{\text{W}}}\left|\lambda\mathbb{E}[\xi f(W+\xi)]-\mathbb{E}[Wf(W)]\right|&\leq\sup_{h\in\mathcal{H}_{\text{W}}}\lVert\Delta f\rVert\left(\sigma^2-\sum_{k=1}^n\mathbb{E}[Y_k^2]+2\sum_{k=1}^n\mathbb{E}[Y_k]^2\right)\\
&\leq\sqrt{\frac{2}{\pi}}\sigma^{-1}\left(\sum_{i\not=j}\text{Cov}(Y_i,Y_j)+\sum_{k=1}^n\mathbb{E}[Y_k]^2\right)\,,
\end{align*}
again by \eqref{eq:GaussSteinFactors}, and noting that the covariances are non-negative in this case. Proceeding as in the proof of Theorem \ref{thm:negass}, Theorem \ref{thm:Gauss} then gives us
\[
d_\text{W}(W,Z)\leq\frac{\sum_{k=1}^n\mathbb{E}[Y_k^3]}{\sum_{k=1}^n\mathbb{E}[Y_k^2]}
+\sqrt{\frac{8}{\pi}}\sigma^{-1}\left(\sum_{i\not=j}\text{Cov}(Y_i,Y_j)+\sum_{k=1}^n\mathbb{E}[Y_k]^2\right)\,,
\]
and the conclusion follows.
\end{proof}
In their Theorem 1.3, Goldstein and Wiroonsri \cite{goldstein18} show that if $\widetilde{W}=\widetilde{Y}_1+\cdots+\widetilde{Y}_n$ is a sum of mean-zero, associated random variables with $|\widetilde{Y}_i|\leq B$ for each $i$ and with $\text{Var}(\widetilde{W})=1$, then
\begin{equation*}
d_\text{W}(\widetilde{W},\widetilde{Z})\leq5B+\sqrt{\frac{8}{\pi}}\sum_{i\not=j}\text{Cov}(\widetilde{Y_i},\widetilde{Y}_j)\,.
\end{equation*}
Similar remarks apply as in the negatively associated case when comparing this result to our Theorem \ref{thm:ass}. Goldstein and Wiroonsri do not need that the underlying random variables are integer-valued, but do require a boundedness condition not needed in our Theorem \ref{thm:ass}.

\subsection{Application to simple random sampling}\label{subsec:normal_sample}

Given (not necessarily distinct) non-negative integers $c_1,\ldots,c_m$, we take a sample of size $n<m$ without replacement. Let $Y_1,\ldots,Y_n$ denote the sampled values, and $W=Y_1+\cdots+Y_n$ their sum. It is known that $Y_1\ldots,Y_n$ are negatively associated; see Section 3.2 of \cite{joagdev83}. It is easily shown that $\mathbb{E}[Y_k^\alpha]=m^{-1}\sum_{j=1}^mc_j^\alpha$ for each $k$ and
\[
\sum_{i\not=j}\text{Cov}(Y_i,Y_j)=\frac{n(n-1)}{m(m-1)}\sum_{i=1}^m\sum_{j\not=i}c_ic_j-\frac{n(n-1)}{m^2}\left(\sum_{j=1}^mc_j\right)^2\,,
\]
so that $\theta=\mathbb{E}W=\frac{n}{m}\sum_{j=1}^mc_j$ and
\[
\sigma^2=\text{Var}(W)=\frac{n}{m}\sum_{j=1}^mc_j^2+\frac{n(n-1)}{m(m-1)}\sum_{i=1}^m\sum_{j\not=i}c_ic_j-\frac{n^2}{m^2}\left(\sum_{j=1}^mc_j\right)^2\,.
\]
Our Theorem \ref{thm:negass} then gives the following Gaussian approximation bound.
\begin{proposition}
With $W$ as above,
\[
d_\text{W}(\widetilde{W},\widetilde{Z})\leq\frac{\sum_{j=1}^mc_j^3}{\sigma\sum_{j=1}^mc_j^2}+\sqrt{\frac{8}{\pi}}\frac{n}{m}\sigma^{-2}\left(\frac{n}{m}\left(\sum_{j=1}^mc_j\right)^2-\frac{n-1}{m-1}\sum_{i=1}^m\sum_{j\not=i}c_ic_j\right)\,,
\]
where $\widetilde{W}=\sigma^{-1}(W-\theta)$ and $\widetilde{Z}\sim\text{N}(0,1)$.
\end{proposition}
Note that if we were to instead apply \eqref{eq:Wnegass} in this setting we would take
\[
B=\frac{\max_i\{c_i\}-\frac{1}{m}\sum_{j=1}^mc_j}{\sigma}\,,
\]
which would give a rather poor upper bound in the case where some of the $c_j$ are much larger than a typical element of this sequence. Finally, we note that Goldstein and Reinert \cite{goldstein97} also derive a Gaussian approximation bound in this setting in their Theorem 4.1, but consider only a set of test functions with four bounded derivatives, so do not obtain a bound in Wasserstein distance. 

\subsection{Application to urn models with overflow}\label{subsec:overflow}

Suppose $n$ balls are each assigned independently to one of $m$ urns, with each ball assigned to urn $j$ with probability $p_j$ for $j=1,\ldots,m$, where $p_1+\cdots+p_m=1$. Let $S_j$ denote the number of balls assigned to urn $j$. Several statistics of interest in this setting may be written in the form $W=\sum_{j=1}^my(S_j)$ for some function $y:\mathbb{R}\to\mathbb{R}$. Examples include those considered by Boutsikas and Koutras \cite{boutsikas02b}, $y(x)=\mathbf{1}(x\geq k)$ and $y(x)=(x-k+1)\mathbf{1}(x\geq k)$. These are motivated by a setting in which each urn has capacity for $k-1$ balls, and then $W$ counts the number of urns exceeding this capacity and the number of excess balls assigned to these overflowing urns, respectively. In each of these cases $y$ is an increasing function, and $y(S_1),\ldots,y(S_m)$ are identically distributed. Since $S_1,\ldots,S_m$ are negatively associated, then when $y$ is increasing we also have that $y(S_1),\ldots,y(S_m)$ are negatively associated. We may thus apply our Theorem \ref{thm:negass} to this setting.

In the case where $p_j=p=1/m$ for each $j=1,\ldots,m$, and for each of the two cases of $y$ above, Boutsikas and Koutras \cite{boutsikas02b} establish a Poisson or compound Poisson limit for $W$ and derive expressions for the mass function of $y(S_1)$ and the covariance $\text{Cov}(y(S_1),y(S_2))$, thus providing the ingredients we need to immediately apply our Theorem \ref{thm:negass} to also yield a Gaussian approximation result. 

For example, consider the case where we count the total number of balls assigned to urns that have already reached their capacity, $W=\sum_{j=1}^mY_j$, where $Y_j=(S_j-k+1)\mathbf{1}(S_j\geq k)$. In this setting we have 
\[
\mathbb{P}(Y_j=l)=\binom{n}{l+k-1}p^{l+k-1}(1-p)^{n-l-k+1}\,,
\]
for $l=1,\ldots,n-k+1$, and
\begin{multline*}
\text{Cov}(Y_1,Y_2)=\sum_{i=k}^{n-k}\sum_{j=k}^{n-i}(i-k+1)(j-k+1)\binom{n}{i,j}p^{i+j}(1-2p)^{n-i-j}\\
-\left(\sum_{i=k}^n(i-k+1)\binom{n}{i}p^i(1-p)^{n-i}\right)^2\,,
\end{multline*}
where $\binom{n}{i,j}=\frac{n!}{i!j!(n-i-j)!}$ if $i+j\leq n$, and is zero otherwise; see page 278 of \cite{boutsikas02b}. Our Theorem \ref{thm:negass} gives the following.
\begin{proposition}
With $W$ as above,
\[
d_\text{W}(\widetilde{W},\widetilde{Z})\leq\frac{\mathbb{E}[Y_1^3]}{\sigma\mathbb{E}[Y_1^2]}+\sqrt{\frac{8}{\pi}}m\sigma^{-2}\left((1-m)\text{Cov}(Y_1,Y_2)+\mathbb{E}[Y_1]^2\right)\,,
\]
where $\sigma^2=\text{Var}(W)$. 
\end{proposition}
It is clear that using \eqref{eq:Wnegass} would  give a rather poor bound here, since the $Y_j$ are bounded only by the rather large value $n-k+1$, values close to which are seen only with vanishingly small probability; see also Table 6 of \cite{boutsikas02b}, where the distribution function of $W$ is evaluated numerically in several examples. Finally, we also refer the interested reader to Section 4.2 of \cite{chen10}, in which the authors establish a more general Gaussian approximation bound in Kolmogorov distance for statistics of the type we consider here, albeit with rather large constants in the error bound.

\section{Gaussian approximation with a vanishing third moment}\label{sec:ThirdMom}

In this section we move away from our previous compound Poisson setting, and consider a bound in approximation by a standard Gaussian distribution under the assumption that the random variable $W$ being approximated is infinitely divisible and has vanishing third moment. As in Section \ref{sec:pois}, this is a more direct application of biasing with an independent increment: we use the decomposition of the zero-biased version, $W^\text{z}$, as $W+Y$ for some independent random variable $Y$ to give a convex ordering between $W$ and $W^\text{z}$, and hence a bound on $\mathbb{E}h(W)$ for suitable functions $h$. Our main result is the following.
\begin{theorem}\label{thm:ThirdMom}
Let $W$ be an infinitely divisible random variable with $\mathbb{E}[W]=\mathbb{E}[W^3]=0$, $\mathbb{E}[W^2]=1$, and with finite fourth moment. Let $h:\mathbb{R}\to\mathbb{R}$ be twice differentiable with absolutely continuous first derivative. Then
\[
|\mathbb{E}h(W)-\mathbb{E}h(Z)|\leq\frac{1}{3}\lVert h^{\prime\prime}\rVert\left(\mathbb{E}[W^4]-3\right)\,,
\]
where $Z\sim\text{N}(0,1)$.
\end{theorem}
\begin{proof}
We will again use Stein's method for Gaussian approximation, and to that end let $f:\mathbb{R}\to\mathbb{R}$ solve \eqref{eq:GaussStein} with $\theta=0$, $\sigma^2=1$, and $h$ as in the statement of the theorem. We will make use of the following bound on the third derivative of this $f$, given by Theorem 1.1 of \cite{daly08}:
\begin{equation}\label{eq:GaussSteinFactorII}
\lVert f^{(3)}\rVert\leq2\lVert h^{\prime\prime}\rVert\,.
\end{equation}
For the remainder of the proof we write $F$ for the distribution function of $W$, and $G$ for that of $W^\text{z}$. We now use the Stein equation \eqref{eq:GaussStein}, the definition \eqref{eq:zb_def} and integrate by parts twice to write
\begin{align*}
\mathbb{E}h(W)-\mathbb{E}h(Z)&=\mathbb{E}[f^\prime(W)-Wf(W)]
=\mathbb{E}[f^\prime(W)-f^\prime(W^\text{z})]\\
&=\int_{-\infty}^\infty f^{\prime\prime}(x)\left[G(x)-F(x)\right]\,\text{d}x
=\int_{-\infty}^\infty f^{(3)}(x)\int_{-\infty}^x\left[F(y)-G(y)\right]\,\text{d}y\,\text{d}x\,,
\end{align*}
where in the final equality we use the definition \eqref{eq:zb_def} to note that since $\mathbb{E}[W^\text{z}]=\frac{1}{2}\mathbb{E}[W^3]$ we have $\mathbb{E}[W^\text{z}]=\mathbb{E}[W]=0$.

Now, the infinite divisibility of $W$ gives us that $W^\text{z}$ is equal in distribution to $W+Y$ for some random variable $Y$ independent of $W$. The vanishing third moment of $W$ means that $\mathbb{E}[W^\text{z}]=0$ and hence $\mathbb{E}[Y]=0$. From this, Theorem 3.A.34 of \cite{shaked07} gives us that $W$ precedes $W^\text{z}$ in the usual convex order, and Theorem 3.A.1 of \cite{shaked07} then implies that $\int_{-\infty}^xF(y)\,\text{d}y\leq\int_{-\infty}^x G(y)\,\text{d}y$ for each $x\in\mathbb{R}$. Hence,
\begin{align*}
|\mathbb{E}h(W)-\mathbb{E}h(Z)|&\leq\lVert f^{(3)}\rVert\int_{-\infty}^\infty\left|\int_{-\infty}^x[F(y)-G(y)]\,\text{d}y\right|\,\text{d}x\\
&=\lVert f^{(3)}\rVert\int_{-\infty}^\infty\int_{-\infty}^x[G(y)-F(y)]\,\text{d}y\,\text{d}x\,.
\end{align*}
Combining Theorem 2 of Boutsikas and Vaggelatou \cite{boutsikas02} with their equation (6), we thus have that 
\[
|\mathbb{E}h(W)-\mathbb{E}h(Z)|\leq\frac{1}{2}\lVert f^{(3)}\rVert\left(\text{Var}(W^\text{z})-\text{Var}(W)\right)\,.
\]
The proof is complete on using the bound \eqref{eq:GaussSteinFactorII} and noting that the definition \eqref{eq:zb_def} gives $\text{Var}(W^\text{z})=\frac{1}{3}\mathbb{E}[W^4]$.
\end{proof}
Under the assumption that $W$ is infinitely divisible and has vanishing third moment, Theorem \ref{thm:ThirdMom} gives us an upper bound that depends on the difference between the fourth moment of $W$ and that of a standard Gaussian. This is superior to other Gaussian approximation bounds available in the case of a vanishing third moment, such as that given in Corollary 3.3 of \cite{gaunt16}, where the upper bound is a sum of positive contributions from second- and fourth-moment terms.

We note also that the constant in the upper bound of Theorem \ref{thm:ThirdMom} could be improved at the cost of requiring stronger differentiability conditions on $h$. In the proof of our result we could replace the upper bound in \eqref{eq:GaussSteinFactorII} by either $\frac{1}{\sqrt{2}\Gamma(5/2)}\lVert h^{(3)}\rVert$ or $\frac{1}{4}\lVert h^{(4)}\rVert$ under the assumption that the higher-order derivative exists; see \cite{gaunt25} for a discussion.

\subsection{Example: Gaussian approximation for Student's \emph{t} distribution}

We conclude this section with a brief illustrative example. Let $T\sim\text{t}(m)$ have Student's \emph{t} distribution with $m>4$ degrees of freedom, so that the fourth moment of $T$ is finite. That is, $T$ has density proportional to $(1+\frac{x^2}{m})^{-\frac{1}{2}(m+1)}$ for $x\in\mathbb{R}$, with $\mathbb{E}[T]=\mathbb{E}[T^3]=0$, $\mathbb{E}[T^2]=\frac{m}{m-2}$ and $\mathbb{E}[T^4]=\frac{3m^2}{(m-2)(m-4)}$. We further know that $T$ is infinitely divisible; see Remark 8.12 of \cite{sato99}. Normalising to have unit variance, we choose $W=\sqrt{\frac{m-2}{m}}T$ so that
\[
\mathbb{E}[W^4]=\left(\frac{m-2}{m}\right)^2\frac{3m^2}{(m-2)(m-4)}=\frac{3(m-2)}{m-4}\,.
\]
An application of Theorem \ref{thm:ThirdMom} then gives the following bound.
\begin{proposition}
With $W$ as above and $h:\mathbb{R}\to\mathbb{R}$ as in Theorem \ref{thm:ThirdMom},
\[
|\mathbb{E}h(W)-\mathbb{E}h(Z)|\leq\frac{2}{m-4}\lVert h^{\prime\prime}\rVert\,,
\]
where $Z\sim\text{N}(0,1)$.
\end{proposition}

\subsubsection*{Acknowledgements}
The author thanks anonymous reviewers and an Editor for their constructive comments and suggestions that improved the exposition in the present work.


\begin{thebibliography}{99}

\bibitem{arras19} B.~Arras and C.~Houdr\'e (2019). \emph{On Stein's Method for Infinitely Divisible Laws with Finite First Moment}. Springer, Cham.

\bibitem{arratia19} R.~Arratia, L.~Goldstein and F.~Kochman (2019). Size bias for one and all. \emph{Probab. Surv.} {\bf 16}: 1--61. 

\bibitem{balabdaoui19} F.~Balabdaoui, G.~de~Fournas-Labrosse and J.~Giguelay (2019). Multiple monotonicity of discrete distributions: The case of the Poisson model. \emph{Electron. J. Stat.} {\bf 13}: 1744--1758.

\bibitem{barbour92_paper} A.~D.~Barbour, L.~H.~Y.~Chen and W.-L.~Loh (1992). Compound Poisson approximation for nonnegative random variables via Stein's method. \emph{Ann. Probab.} {\bf 20}(4): 1843--1866. 

\bibitem{barbour92_book} A.~D.~Barbour, L.~Holst and S.~Janson (1992). \emph{Poisson Approximation}. Oxford University Press, Oxford.

\bibitem{barbour98} A.~D.~Barbour and S.~Utev (1998). Solving the Stein equation in compound Poisson approximation. \emph{Adv. in Appl. Probab.} {\bf 30}(2): 449--475.

\bibitem{barbour99} A.~D.~Barbour and A.~Xia (1999). Poisson perturbations. \emph{ESAIM Probab. Stat.} {\bf 3}: 131--150. 

\bibitem{barbour00} A.~D.~Barbour and A.~Xia (2000). Estimating Stein's constants for compound Poisson approximation. \emph{Bernoulli} {\bf 6}(4): 581--590. 

\bibitem{bhattacharjee19} C.~Bhattacharjee and L.~Goldstein (2019). Dickman approximation in simulation, summations and perpetuities. \emph{Bernoulli} {\bf 25}(4A): 2758--2792.

\bibitem{bhattacharjee22} C.~Bhattacharjee and M.~Schulte (2022). Dickman approximation of weighted random sums in the Kolmogorov distance. Preprint: available at \texttt{https://arxiv.org/abs/2211.10171}. 

\bibitem{boutsikas02b} M.~V.~Boutsikas and M.~V.~Koutras (2002). On the number of overflown urns and excess balls in an allocation model with limited urn capacity. \emph{J. Statist. Plann. Inference} {\bf 104}(2): 259--286.

\bibitem{boutsikas02} M.~V.~Boutsikas and E.~Vaggelatou (2002). On the distance between convex-ordered random variables, with applications. \emph{Adv. in Appl. Probab.} {\bf 34}(2): 349--374.

\bibitem{chen10} L.~H.~Y.~Chen and A.~R\"ollin (2010). Stein couplings for normal approximation. Preprint, available at \texttt{arXiv:1003.6039}.

\bibitem{chen11} L.~H.~Y.~Chen, L.~Goldstein and Q.-M.~Shao (2011). \emph{Normal approximation by Stein’s method}. Springer, Berlin.

\bibitem{daly08} F.~Daly (2008). Upper bounds for Stein-type operators. \emph{Electron. J. Probab.} {\bf 13}(20): 566--587.

\bibitem{daly13} F.~Daly (2013). Compound Poisson approximation with association or negative association via Stein's method. \emph{Electron. Commun. Probab.} {\bf 18}(30): 1--12.

\bibitem{daly17} F.~Daly (2017). On magic factors in Stein's method for compound Poisson
approximation. \emph{Electron. Commun. Probab.} {\bf 22}(67): 1--10.

\bibitem{daly22} F.~Daly (2022). Gamma, Gaussian and Poisson approximations for random sums using size-biased and generalized zero-biased couplings. \emph{Scand. Actuar. J.} {\bf 2022}(6): 471--487.

\bibitem{dobler17} C.~D\"obler (2017). Distributional transformations without orthogonality relations. \emph{J. Theoret. Probab.} {\bf 30}: 85--116.

\bibitem{esary67} J.~D.~Esary, F.~Proschan and D.~W.~Walkup (1967). Association of random variables, with applications. \emph{Ann. Math. Statist.} {\bf 44}: 1466--1474.

\bibitem{gaunt16} R.~E.~Gaunt (2016). Rates of convergence in normal approximation under moment conditions via new bounds on solutions of the Stein equation. \emph{J. Theoret. Probab.} {\bf 29}: 231--247.

\bibitem{gaunt25} R.~E.~Gaunt (2025). On Stein factors in Stein’s method for normal approximation. \emph{Statist. Probab. Lett.} {\bf 219}: 110339.

\bibitem{goldstein97} L.~Goldstein and G.~Reinert (1997). Stein's method and the zero bias transformation with application to simple random sampling. \emph{Ann. Appl. Probab.} {\bf 7}(4): 935--952.

\bibitem{goldstein18} L.~Goldstein and N.~Wiroonsri (2018). Stein's method for positively associated random variables with applications to the Ising and voter models, bond percolation, and contact process. \emph{Ann. Inst. Henri Poincar\'e Probab. Stat.} {\bf 54}(1): 385--421.

\bibitem{goldstein24} L.~Goldstein and T.~Kemp (2024). Free zero bias and infinite divisibility. Preprint, available at \texttt{arXiv:2403.19860}.

\bibitem{grubel05} R.~Gr\"ubel and N.~Stefanoski (2005). Mixed Poisson approximation of node depth distributions in random binary search trees. \emph{Ann. Appl. Probab.} {\bf 15}(1A): 279--297.

\bibitem{joagdev83} K.~Joag-Dev and F.~Proschan (1983). Negative association of random variables, with applications. \emph{Ann. Statist.} {\bf 11}: 286--295.

\bibitem{karlis05} D.~Karlis and E.~Xekalaki (2005). Mixed Poisson distributions. \emph{Int. Stat. Rev.} {\bf 73}(1): 35--58.

\bibitem{ross11} N.~Ross (2011). Fundamentals of Stein's method. \emph{Probab. Surv.} {\bf 8}: 210--293.

\bibitem{sato99} K.-I.~Sato (1999). \emph{L\'evy Processes and Infinitely Divisible Distributions}. Cambridge Studies in Advanced Mathematics 68, Cambridge University Press, Cambridge.

\bibitem{shaked07} M.~Shaked and J.~G.~Shanthikumar (2007). \emph{Stochastic Orders}. Springer, New York.

\bibitem{steutel73} F.~W.~Steutel (1973). Some recent results in infinite divisibility. \emph{Stochastic Process. Appl.} {\bf 1}: 125--143.

\bibitem{wiroonsri18} N.~Wiroonsri (2018). Stein's method for negatively associated random variables with applications to second-order stationary random fields. \emph{J. Appl. Probab.} {\bf 55}(1): 196--215.

\end{thebibliography}
\end{document}